\documentclass[12pt]{amsart}
\usepackage{fullpage}

\usepackage{amssymb,amsmath}
\usepackage{tikz}
\usepackage{tikz-cd}
\usepackage{amssymb, amsmath, amsfonts, amsthm, graphics}
\usepackage[all]{xy}
\usepackage{amsrefs}

\usetikzlibrary{arrows,decorations.pathmorphing,decorations.pathreplacing,positioning,shapes.geometric,shapes.misc,decorations.markings,decorations.fractals,calc,patterns}

\usepackage{enumitem}

\usepackage[bottom]{footmisc}

\newcommand{\into}{\hookrightarrow}

\theoremstyle{plain}
\newtheorem{Theorem}{Theorem}
\newtheorem{Proposition}[Theorem]{Proposition}
\newtheorem{Corollary}[Theorem]{Corollary}
\newtheorem{Lemma}[Theorem]{Lemma}

\theoremstyle{definition}
\newtheorem{Definition}[Theorem]{Definition}

\newtheorem{Example}[Theorem]{Example}

\theoremstyle{remark}

\numberwithin{Theorem}{section} 

\DeclareMathOperator{\Hom}{Hom}

\DeclareMathOperator{\Tor}{Tor}

\DeclareMathOperator{\Obj}{Obj}
\DeclareMathOperator{\im}{im}

\DeclareMathOperator{\rk}{rk}

\newcommand{\II}{\mathbb{I}}

\newcommand{\KK}{\mathbb{K}}

\newcommand{\NN}{\mathbb{N}}

\newcommand{\ZZ}{\mathbb{Z}}

\newcommand{\cA}{\mathcal{A}}
\newcommand{\cB}{\mathcal{B}}
\newcommand{\cC}{\mathcal{C}}
\newcommand{\cD}{\mathcal{D}}
\newcommand{\cE}{\mathcal{E}}
\newcommand{\cF}{\mathcal{F}}
\newcommand{\cG}{\mathcal{G}}
\newcommand{\cH}{\mathcal{H}}
\newcommand{\cI}{\mathcal{I}}
\newcommand{\cJ}{\mathcal{J}}
\newcommand{\cK}{\mathcal{K}}

\newcommand{\cM}{\mathcal{M}}
\newcommand{\cN}{\mathcal{N}}

\newcommand{\cP}{\mathcal{P}}

\newcommand{\cR}{\mathcal{R}}
\newcommand{\cS}{\mathcal{S}}
\newcommand{\cT}{\mathcal{T}}

\newcommand{\cV}{\mathcal{V}}

\newcommand{\Id}{\operatorname{\it Id}}

\newcommand{\R}{K[x_1, \ldots, x_n]}

\newcommand{\W}{\Omega}
\newcommand{\WV}{\Omega_V}
\newcommand{\WVp}{\Omega_{V'}}
\newcommand{\WVD}{\Omega_{V,d}}
\newcommand{\WN}{\WVD}
\newcommand{\WVE}{\Omega_{V,e}}

\newcommand{\WVPD}{\Omega_{V',d}}
\newcommand{\WNp}{\WVPD}

\newcommand{\wv}{W \otimes V}

\DeclareMathOperator{\GL}{GL}
\newcommand{\glv}{\GL(V)}

\newcommand{\blank}{ \rule[0.1cm]{0.3cm}{0.1pt}}
\newcommand{\Blank}{ \rule[0.1cm]{0.4cm}{0.1pt}}

\newcommand{\ten}{\blank \otimes V}
\newcommand{\TV}{\cT_V}
\newcommand{\PV}{\cP_V}

\newcommand{\PD}{\cP_d}
\newcommand{\pd}{\PD}
\newcommand{\PVD}{\cP_{V,d}}
\newcommand{\pv}{\PVD}

\newcommand{\HVD}{\cH_{V,d}}

\newcommand{\Exterior}{\mathchoice{{\textstyle\bigwedge}}%
    {{\bigwedge}}%
    {{\textstyle\wedge}}%
    {{\scriptstyle\wedge}}}
\newcommand{\we}{\Exterior}

\newcommand{\Sl}{\cS_{\lambda}}
\newcommand{\Slp}{\cS_{\lambda'}}
\newcommand{\Sm}{\cS_{\mu}}
\newcommand{\Smp}{\cS_{\mu'}}
\newcommand{\Sn}{\cS_{\nu}}

\newcommand{\clw}{c_{\mu \nu}^{\lambda}}

\newcommand{\ml}{m_{\lambda}}

\newcommand{\rep}{\mathbf{Rep}_V}
\newcommand{\ve}{ \mathbf{Vec}}
\newcommand{\poly}{\mathbf{Poly}}
\newcommand{\polyd}{\poly_d}
\newcommand{\gp}{\mathbf{GPoly}}
\newcommand{\gv}{ \mathbf{GVec}}

\DeclareMathOperator{\F}{{\bf Fun}}

\newcommand{\fvr}{\F(\ve,\rep)}
\newcommand{\frr}{\F(\rep,\rep)}

\DeclareMathOperator{\sym}{Sym}

\newcommand{\mm}{\mathfrak{m}}

\DeclareMathOperator{\reg}{reg}

\newcommand {\ia}{\cI_{\cA}}

\newcommand {\ja}{\cJ_{\cA}}

\title{Resolutions of ideals associated to\\ subspace arrangements}
\date{}
\author{Francesca Gandini \\ fragandi@umich.edu}

\begin{document}

\maketitle
\begin{abstract}
Given a collection of $t$ subspaces in an $n$-dimensional $\mathbb{K} $-vector space $W$ we can associate to them $t$ vanishing ideals in the symmetric algebra $\mathcal{S}(W^*) = \KK[x_1,x_2,\dots,x_n]$. As a subspace is defined by a set of linear equations, its vanishing ideal is generated by linear forms so it is a \emph{linear ideal}. Conca and Herzog showed that the Castelnuovo-Mumford regularity of the product of $t$ linear ideals is equal to $t$. Derksen and Sidman showed that the Castelnuovo-Mumford regularity of the intersection of $t$ linear ideals is at most $t$ and they also showed that similar results hold for a more general class of ideals constructed from linear ideals. In this paper we show that analogous results hold when we replace the symmetric algebra $\mathcal{S}(W^*)$ with the exterior algebra $ \bigwedge(W^*)$ and work over a field of characteristic 0. To prove these results we rely on the functoriality of free resolutions and construct a functor $\Omega$ from the category of polynomial functors to itself. The functor $\Omega$  transforms resolutions of ideals in the symmetric algebra to resolutions of ideals in the exterior algebra.

\vspace{12pt}

 	

\end{abstract}

{\let\thefootnote\relax\footnote{{The author was partly supported by NSF grant DMS-1601229.}}}

\section{Introduction}

By a subspace arrangement we mean a finite collection of subspaces in Euclidean space. Questions about the complement of a real hyperplane arrangement date back to the mid-1800's, whilst the more recent trend of research investigates general subspace arrangements in combinatorics, topology, and complexity theory (see the survey \cite{b}). In this study we investigate these objects from an algebraic perspective. There are two main types of algebraic structures associated to a subspace arrangement: the cohomology ring of the complement of a hyperplane arrangement and the vanishing ideal of a subspace arrangement. Both types are discussed in the survey \cite{ss}. Here we study topics related to the vanishing ideal of a subspace arrangement.

In 1999 Derksen conjectured that
the vanishing ideal of a union of $t$ subspaces is generated by polynomials of degree at most $t$. He used this conjecture on subspace arrangements to establish a bound on the degree of invariants of finite groups. Specifically, he proved that in the non-modular case (when the group order does not divide the characteristic of the base field) Noether's degree bound holds if the conjecture holds for $t=|G|$. Bernd Sturmfels made an even stronger conjecture: the vanishing ideal of a union of $t$ subspaces has Castelnuovo-Mumford regularity at most $t$.
Derksen's result on the connection between invariants and subspace arrangements sparked our interest in studying ideals associated to subspace arrangements to prove results in invariant theory. In this paper, we study ideals of subspace arrangements over the exterior algebra via their connection with ideals over the symmetric algebra. In a later publication we will connect these results to the context of non-commutative invariant theory. In particular, in her thesis \cite{fg} the author proved an analog of Noether's Degree Bound \cite{em} for invariant skew polynomials in characteristic zero.

Suppose that $W_1,W_2,\dots,W_t$ are subspaces of an $n$-dimensional $\KK$-vector space $W\cong \KK^n$ and let $I_1,I_2,\dots,I_t\subseteq \KK[x_1,x_2,\dots,x_n]$ be the vanishing ideals of $W_1,W_2,\dots,W_t$. These vanishing ideals are \emph{linear} ideals in the sense that they are generated by linear forms. Conca and Herzog showed that the Castelnuovo-Mumford regularity of the product ideal $I_1I_2\cdots I_t$ is equal to $t$ (see~\cite{ch}). Derksen and Sidman proved Sturmfels' conjecture, namely they showed that the Castelnuovo-Mumford regularity of the intersection ideal $I_1\cap I_2\cap \cdots\cap I_t$ is at most $t$ (see~\cite{ds1}); similar results hold for more general ideals constructed from linear ideals (see~\cite{ds2}). Because it is possible to use the regularity of an ideal to bound the degree of its generators, then a regularity result yields a degree bound for the generators. 

Our approach is to study the product and the intersection of linear ideals over the exterior algebra. Over the symmetric algebra $ \cS(W^*) = \KK[x_1,x_2,\dots,x_n]$, we have good bounds on the Castelnuovo-Mumford regularity (hereafter just referred to as regularity). We leverage these results for the symmetric algebra $ \cS(W^*)$ to prove similar regularity bounds over the exterior algebra $ \we(W^*)$. 

In the literature, monomial and square-free ideals over the exterior algebra have been studied in relation to their analogues in the symmetric algebra. In particular, monomial ideals in the exterior algebra have been studied in \cite{aa}. Using square-free modules in the exterior algebra, one can define a generalization of Alexander's duality (see \cite{r}). In the context of hyperplane arrangements, the homology and the cohomology rings of the complement of the arrangement are modules over the exterior algebra and have been studied in \cite{e}. These results rely on the idea of creating a connection between resolutions over the symmetric algebra and resolutions over the exterior algebra. Our approach also relies on a similar idea, even though it exploits a different method: a functor on polynomial functors.

Our methods allow us to study any finite wedge product of linear ideals in the exterior algebra. In particular, we have the following result.
\begin{Theorem}\label{nice}
Assume that $V$ is a finite-dimensional vector space over a field of characteristic zero. In the exterior algebra $\we (V)$, the wedge product of a finite number of linear ideals has a linear resolution. 
\end{Theorem}
Specifically, this theorem is a direct consequence of our main result, Theorem~\ref{bignice}, which establishes that the wedge product of $t$ linear ideals is $t$-regular. In general, we are interested in computing the regularity of a module because this numerical invariant gives us a measure of its complexity. Even for ideals that are simple to describe, it can be hard to explicitly compute their regularity. Moreover, even prime ideals can have very large regularity as shown in \cite{mp} by McCullough and Peeva's counterexamples to the Eisenbud-Goto Regularity Conjecture \cite{eg}. Our result shows that working with ideals constructed from linear ideals, we have the best possible regularity bound irrespective of whether we work over the symmetric algebra or the exterior algebra. 

To study ideals in the exterior algebra, we construct a way to transfer information between the symmetric algebra and the exterior algebra. To this goal, we consider ideals of subspace arrangements that are stable under the action of the general linear group and study them using the tools of representation theory. Specifically, we describe a functor, $\W$, on the category of graded polynomial functors.  The functor $\W$ is the transpose functor used by Sam and Snowden \cite{tca, sam1, sam2} to study modules over twisted commutative algebras. The functor $\W$ will transfer homological properties from equivariant resolutions over the symmetric algebra to equivariant resolutions over the exterior algebra.

Ideals with the additional structure of a group representation exhibit interesting behavior even in simple examples. 
In fact, the Hilbert series of the vanishing ideal of a hyperplane arrangement of $d$ hyperplanes in $n$-dimensional space is just $t^d/(1-t)^n$. However, the Hilbert series of an ideal which is stable under a group action is a much more interesting object. In fact, one can define the notion of equivariant Hilbert series of $\glv$-equivariant ideals. We compute equivariant Hilbert series of ideals of subspace arrangements using the combinatorial formula from \cite{sym} and use these computations to write down equivariant resolutions of ideals associated to subspace arrangements. The resolutions considered will be $\glv$-equivariant, meaning that all modules in the resolution will be $\glv$-representations and all maps in the resolution will be maps of $\glv$-representations.

We begin this paper by introducing some background on polynomial functors. In particular, we define the category $\gp$ in which we will operate and we introduce the notion of algebra and module functors in this category. In Sections 3-6 we proceed in the technical construction of the functor $\W$ on the category $\gp$. In Section 7-8 we introduce resolutions in $\gp$ and we study the effect of $\W$ on these resolutions. We also define Castelnuovo-Mumford regularity in this context and show that applying the functor $\W$ does not change the value of this homological invariant. Finally, in Sections 9-10 we introduce the module functors of a subspace arrangement, prove the main result, and provide some examples of computations. 

\subsection*{Acknowledgments}
The author acknowledges the patience and insights of her thesis advisor Harm Derksen. She is also thankful to David Eisenbud for providing several helpful references that were included in the above introduction and to Andrew Snowden for explaining the connection with the work of Sam and Snowden.

\section{Polynomial functors}

In our discussion of polynomial functors we follow the classical treatment of Macdonald \cite{mac}. 
Let us fix a field $\KK$ of characteristic 0.
Let us denote by $\ve$ the category of finite dimensional $\KK$-vector spaces whose morphisms are  the $\KK$-linear maps. This abelian category also has a tensor product, which makes $\ve$ into a symmetric monoidal category.

\begin{Definition}
A functor $\cF$ from $\ve$ to $\ve$ is a polynomial functor if  the map 
$$\cF:\Hom(X,Y) \to \Hom(\cF(X),\cF(Y))$$ is a polynomial mapping for all finite dimensional $\KK$-vector spaces $X,Y$. We say that $\cF$ is homogeneous of degree $d$ if $\cF(\lambda h)=\lambda^d \cF(h)$ for every linear map $h \in \Hom(X,Y)$ and every scalar $\lambda\in \KK$.
\end{Definition}

Let $\cF$ be a polynomial functor. 
We will consider the category of polynomial functors $\poly$. The morphisms in $\poly$ are natural transformations of functors. 
If $\cF$ and $\cG$ are polynomial functors, we define the direct sum functor $\cF \oplus \cG :\ve\to \ve$ by
$(\cF\oplus \cG)(X)=\cF(X)\oplus \cG(X)$ for every finite dimensional vector space $X \in \Obj(\ve)$, and 
$$(\cF\oplus \cG)(h)=\begin{pmatrix}\cF(h)& 0 \\0 & \cG(h)\end{pmatrix}\in \Hom(\cF(X)\oplus \cG(X),\cF(Y)\oplus \cG(Y))$$
 for every linear map $h:X\to Y$.
 We can also define the tensor product of two polynomial functors $\cF$ and $\cG$ by $(\cF\otimes \cG)(X)=\cF(X)\otimes \cG(X)$ for every
 finite dimensional vector space and $(\cF\otimes \cG)(h)=\cF(h)\otimes \cG(h):\cF(X)\otimes \cG(X)\to \cF(Y)\otimes \cG(Y)$ for any linear map $h:X\to Y$.
 This makes $\poly$ into an abelian symmetric monoidal category. If $\cF$ and $\cG$ are homogeneous polynomial functors of degree $d$ and $e$ respectively,
 then $\cF\otimes \cG$ is homogeneous of degree $d+e$.

 For categories {\bf A} and {\bf B} we denote the category of all functors from {\bf A} to {\bf B} by $\F({\bf A},{\bf B})$. Morphisms in $\F({\bf A},{\bf B})$ are natural transformations.
 We can view $\poly$ as a subcategory of $\F(\ve,\ve)$.
 
 For an $n$-dimensional vector space $V$, let $\GL(V)\subseteq \Hom(V,V)$ be the group of invertible linear maps from $V$ to $V$. A polynomial functor $\cF$ gives a polynomial map
 $\Hom(V,V)\to \Hom(\cF (V),\cF (V))$ that restricts to a group homomorphism $\rho:\GL(V)\to \GL(\cF(V))$. This means that $\cF(V)$ is a polynomial representation of $\GL(V)$.

 A partition is a sequence $\lambda=(\lambda_1,\lambda_2,\dots,\lambda_r)$ of positive integers with $\lambda_1\geq \lambda_2\geq \cdots\geq \lambda_r$.
 For each partition $\lambda$ one can define a polynomial functor $\cS_\lambda:\ve\to \ve$ that is homogeneous of degree $|\lambda|=\lambda_1+\lambda_2+\cdots+\lambda_r$. For a finite dimensional vector space $V$, the representation $\cS_\lambda (V)$ is an irreducible representation of $\GL(V)$.
  The space $\cS_{(d)}(V)=\sym^d(V)$ is the $d$-th symmetric power of $V$ whilst the space $\cS_{(1,1,\dots,1)}(V)=\cS_{(1^d)}(V)=\we^d(V)$
 is the $d$-th exterior power of $V$.
   It follows from Schur's lemma that 
 $$\Hom(\cS_\lambda,\cS_\mu)=\begin{cases}
 \KK & \mbox{if $\lambda=\mu$;}\\
 0 & \mbox{if $\lambda\neq \mu$.}
 \end{cases}
 $$
 Every polynomial functor is naturally equivalent to a finite direct sum of $\cS_\lambda$'s. 
 By grouping the $\cS_\lambda$'s together we see that
 every  polynomial functor $\cP \in \poly$ is naturally equivalent to a direct sum $\cP= \bigoplus_d \cP_d$, where $\cP_d$ is a homogeneous polynomial functor of degree $d$. We will denote the full subcategory of homogeneous polynomial functors of degree $d$ by $\poly_d$. For more details, the interested reader can consult \cite[p.~150]{mac}.

Let $\rep$ denote the category of finite dimensional rational representations of $\glv$ where the  morphism are $\glv$-equivariant linear maps. 

\begin{Lemma}\label{PV}
A polynomial functor $\cP$ on the category of finite dimensional vector spaces $\ve$ induces a functor $\PV$ on the category of $\glv$-representations $\rep$. 
\end{Lemma}

\begin{proof}
Let us consider a $\glv$-representation $\rho_U:\glv\to \GL(U)$. The polynomial functor $\cP$ gives a polynomial map $\Hom(U,U)\to \Hom(\cP(U),\cP(U))$
which restricts to a representation $\GL(U)\to \GL(\cP(U))$. The composition $\glv\to\GL(U)\to\GL(\cP(U))$ makes $\cP(U)$ into a representation of $\glv$.

 Let $\phi$ be a $\glv$-equivariant map from $U$ to $U'$, so that for all $g \in \glv$ the following diagram commutes:
\[
\begin{tikzcd}
 U \arrow[r, "\phi"] \arrow[d, "\rho_U(g)"] & U' \arrow[d, "\rho_U(g)"] \\
 U  \arrow[r, "\phi"] & U' 
\end{tikzcd} .\]

Applying $\cP$ to this diagram, we notice that the resulting diagram also commutes as 
\[ \cP(\rho_U(g))\cP(\phi) = \cP(\rho_U(g) \phi) = \cP(\phi \rho_U(g)) = \cP(\phi) \cP(\rho_U(g)), \]
by functoriality of $\cP$ and our assumptions on $\phi$. This shows that $\cP(\phi):\cP(U)\to\cP(U')$ is $\glv$-equivariant.
 
 We conclude that  $\cP$ induces a functor from $\rep$ to itself.
\end{proof}

We can consider the category $\poly_V$ of polynomial functors from $\rep$ to itself. Morphisms in the category $\poly_V$ are  $\glv$-equivariant natural transformations. An object $\cP$ is the category $\poly$ induces an object $\cP_V$ in $\poly_V$ by Lemma~\ref{PV}.

\section{The category $\gp$}
We also consider the category $\gv$ of graded vector spaces.
The objects of $\gv$ are graded vector spaces $V=\bigoplus_{d=0}^\infty V_d$ such that $V_d$ is finite dimensional for all $d$.
A morphism $\phi:V\to W$ in the category $\gv$ is a linear map that respects the grading, i.e., $\phi(V_d)\subseteq W_d$ for all $d$. 
The tensor product of two graded vector spaces $V,W$ in $\gv$ is defined by $(V\otimes W)_d=\bigoplus_{e=0}^d V_e\otimes W_{d-e}$.
This makes $\gv$ into a symmetric monoidal category.

Next we describe the full subcategory $\gp$ in the functor category $\F(\ve, \gv)$. 
\begin{Definition}
An object $\cF$ in $\gp$ is a functor in $\F(\ve, \gv)$ with the property that
\[ V \mapsto \cF(V)_d \]
is a homogeneous polynomial functor of degree $d$. Morphisms in $\gp$ are natural transformations.
\end{Definition}

An example of a functor in $\gp$ is the functor $\cS = \sym$, mapping a vector space $V$ to the symmetric algebra $\cS(V)= \sym(V)$ on $V$. Similarly, another such functor in $\gp$ is the exterior functor $\we$ that maps a vector space $V$ to its exterior algebra $\we(V)$.

The category $\gp$ is a symmetric monoidal category via the tensor structure inherited from  $\gv$.
In $\gp$ we have that 
\[ ((\cF \otimes \cG) (V))_d = (\cF(V) \otimes \cG(V))_d = \textstyle \bigoplus_{e=0}^d \cF(V)_e \otimes \cG(V)_{d-e}. \] 

We will also view $\KK$ as an object in $\gp$ as the functor that sends every vector space to the graded vector space $\KK$ concentrated in degree 0. The object $\KK$ is the identity in the monoidal category $\gp$. This means that we have a natural equivalence $\kappa:\KK\otimes \cF\to \cF$ for every object $\cF$ in $\gp$. 

\subsection{Algebras and modules in $\gp$}
We will define algebra functors and module functors in $\gp$. These are objects in $\gp$ that satisfy axioms analog to the axioms of algebras and modules, respectively.

\begin{Definition}
An object $\cR$ in  $\gp$ is called an algebra functor if it comes equipped with a multiplication $\mu: \cR \otimes \cR \to \cR$ (i.e.,  a natural transformation of the functor $\cR\otimes \cR$ to the functor $\cR$) and an identity ${\bf 1}:\KK\to \cR$ that satisfy the following axioms.
\begin{description}[font=\normalfont\itshape]
\item[Connected] ${\bf 1}_0:\KK_0 \to \cR_0$ is a natural equivalence. Hence, we assume $\cR_0(V)\cong \KK$ for all vector spaces $V$;
\item[Identity] the following diagram commutes
\[
\xymatrix{
 \KK \otimes \cR \ar[r]^{\kappa} \ar[d]_{\bf{1}\otimes \Id_{\cR}} & \cR \ar[d]^{\Id_{\cR}} \\
\cR\otimes \cR \ar[r]_{\mu} & \cR};
\]
\item[Associative] the following diagram commutes
\[
\xymatrix{
(\cR\otimes \cR)\otimes \cR\ar[rr]^{\cong} \ar[d]_{\mu\otimes\Id_{\cR}} &  & \cR \otimes (\cR \otimes \cR) \ar[d]^{\Id_{\cR} \otimes \mu} \\
\cR\otimes \cR \ar[r]_{\mu} & \cR & \cR \otimes \cR \ar[l]^{\mu} }.
\]
\end{description}
\end{Definition}

We will define a (left) module in a similar fashion.

\begin{Definition}
Given an algebra functor $(\cR, \mu, {\bf 1})$, a left module functor $\cM$ over $\cR$ is an object $\cM$ in $\gp$ equipped with a natural transformation $\nu :  \cR \otimes \cM \to \cM$ that satisfies the following axioms.
\begin{description}[font=\normalfont\itshape]
\item[Identity] the following diagram commutes
\[
\xymatrix{
 \KK\otimes \cM \ar[dr]_{\kappa} \ar[rr]^{{\bf 1}\otimes \Id_{\cM}} & & \cR \otimes \cM \ar[dl]^{\nu} \\
 & \cM & };
\]
\item[Associative] the following diagram commutes
\[
\xymatrix{
(\cR\otimes \cR)\otimes \cM \ar[rr]^{\cong} \ar[d]_{\mu \otimes \Id_{\cM}  } &  & \cR \otimes (\cR \otimes \cM)\ar[d]^{\Id_{\cR}\otimes \nu}\\
\cR \otimes \cM \ar[r]_{\nu} & \cR & \cR \otimes \cM \ar[l]^{\nu} }.
\]
\end{description}

\end{Definition}

Notice that for any vector space $V$ and for every algebra functor $\cR$, we have that $\cR(V)$ is a $\KK$-algebra. Similarly, $\cM(V)$ is a left module over $\cS(V)$. Moreover, the above axioms give us that for every $f \in \Hom(\ve)$, we have that $\cR(f)$ is a homomorphism of $\KK$-algebras.
The symmetric algebra functor $\cS = \sym$ is an example of an algebra functor in $\gp$. If $\cM$ is a left module functor over $\cS$, then for every $V \in \Obj( \ve)$ we have that $\cM(V)$ is a left module over $\cS(V)$. 

\begin{Example}
Consider the symmetric algebra functor $\cS$. For any $n$--dimensional vector space $V$, we have the maximal homogeneous ideal $\cM (V) = (x_1, \ldots, x_n)$ in $\cS(V)= \R$. Thus, in $\gp$ we have a module functor $\cM$ over $\cS$ defined on $\Obj(\ve)$ by $V \mapsto \cM(V)$. Notice that a minimal equivariant resolution for $\cM$ is the Koszul resolution:
\[ \cdots \to \cS \otimes \cS_{(1,1,1)} \to \cS \otimes \cS_{(1,1)} \to \cS \otimes \cS_{(1)} \to \cM \to 0, \] 
an infinite resolution. We will discuss resolutions of modules in $\gp$ at the end of this chapter.
\end{Example}

\subsection{Connections to twisted commutative algebras}
The constructions in this sections are closely related to the notion of twisted commutative algebras studied by Sam and Snowden (\cite{sam1,sam2,tca}).
In particular, our definition of the category $\gp$ is closely related to one interpretation of the category $\cV$ in \cite{tca}. The difference is that we prefer to work with a graded category of polynomial functors and allow for infinite direct sums, rather than considering a category whose objects are representations of $\GL_{\infty}$. 

Algebras in $\gp$ satisfy the same axioms as twisted commutative algebras in the category $\cV$. For this reason, our algebra functors are twisted commutative algebras if one prefers to consider them as objects in the category $\cV$ instead of the category $\gp$. 

Next, we will consider the functor $\W$ on the category $\gp$. The functor $\W$ is the translation of the transpose functor on the category of representations of the symmetric group to the context of representations of the general linear group. Concretely, $\W$ maps the Schur functor $\Sl$ to the Schur functor $\Slp$.

Towards the end of the chapter we will use $\W$ to establish a connection between modules over $\sym$ and modules over $\we$. In \cite{sam2} the authors had already established this connection and used it to prove regularity results.
For example, for a fixed $d$ Sam and Snowden prove that a finitely generated module over the twisted commutative algebra
$V\mapsto \sym(V^d)$ has finite regularity (\cite[Corollary 7.8]{sam2}). Moreover, using their results, one can establish that a finitely generated module over the twisted commutative algebra $V\mapsto \we(V^d)$ also has finite regularity. Furthermore, Snowden used twisted commutative algebras to give bounds to the minimal resolution of invariant rings of finite groups in \cite{sno}.

We include in the next sections a self-contained treatment of the subject, for the benefit of the reader. We start by constructing the functor $\W$ from $\gp$ to itself. One important feature of $\W$ is that $\W(\cS) = \we$. In general, for any homogeneous polynomial functor $\cF_d$ of degree $d$,  $\W (\cF_d)$ will be another homogeneous polynomial functor of degree $d$. In fact, we will first construct  $\W_d$, the $d$th graded piece of $\W$, a functor from the category of homogeneous polynomial functors of degree $d$ to itself. The functor $\W$ can be found in the literature in the context of $\GL_{\infty}$-representations ~\cite[p.~1102]{sam1}. In that context $\W$ is called the transpose functor and it is defined for representations of the infinite symmetric group. The transpose functor is then transferred to $\GL_{\infty}$-representations via Schur-Weyl duality. For the convenience of the reader, we present a construction which does not require previous knowledge of the structure theory of $\GL_{\infty}$ representations.

\section{Definition of $\W_d$ on the category $\poly_d$}
Each object $\cF$ in the category $\gp$ is defined by specifying its graded pieces $\cF_d$. Recall that each $\cF_d$ is required to be some homogeneous polynomial functor of degree $d$. For this reason, we can define a functor $\W$ on $\gp$ by specifying a functor $\W_d$ on the homogeneous polynomial functors of degree $d$ for each degree $d$.

Let $\poly_d$ be the full subcategory of $\poly$ consisting of homogeneous polynomial functors of degree $d$. To be able to define $\W_d$ we will need to go through a multi-step process. The first sections will aim to define a functor $\WVD$ on $\poly_d$ for any fixed vector space $V$. Then we will define $\W_d$ as a direct limit of functors $\WVD$.

We start with a construction from category theory.
For a functor $\cF:\mathbf{A}\to \mathbf{B}$, we can define the functor $\cF^*: \F(\mathbf{B},\mathbf{C}) \to \F(\mathbf{A},\mathbf{C})$ by $\cF^*(\cG)= \cG \circ \cF$, for any functor $\cG:\mathbf{B} \to \mathbf{C}$. Similarly, for any fixed $\cG \in \F(\mathbf{B},\mathbf{C})$, we can define $\cG_*: \F(\mathbf{A},\mathbf{B}) \to \F(\mathbf{A},\mathbf{C})$ by $\cG_*(\cF)= \cG \circ \cF$, for any functor $\cF:\mathbf{A}\to \mathbf{B}$.

We fix a vector space $V$ of dimension $n$. Let us consider the category $\fvr$. We define the functor $\TV:\ve\to\rep$ as the functor $\ten$ which acts by mapping $W \in \Obj(\ve)$ to $\wv \in \Obj(\rep)$. Notice that $\glv$ acts on $\wv$ by trivial action on $W$ and left multiplication on $V$. Let us fix a degree $d$ such that $n \geq d$, where we recall that $n$ is the dimension of the fixed vector space $V$. In the category $\frr$ we consider the full subcategory $\poly_{V,d}$ of homogeneous polynomial functors of degree $d$.

Recall that an object $\cP_d$ in $\poly_d$ induces an object $\cP_{V,d}$ in $\poly_{V,d}$ by Lemma~\ref{PV}.
Finally, let us consider the functor $\HVD:\rep\to\ve$, defined as $\Hom_{\glv}(\we^d(V), \Blank)$. On objects, we have that $\HVD$ maps a $\glv$-representation $U$ to its $\we^d(V)$-isotopic component.
\begin{Definition}
For a polynomial functor $\cP_d$ of degree $c$,  the functor $\WVD(\cP_d):\ve\to\ve$ is defined by 
\[ \WVD(\cP_d) = (\HVD)_*\TV^*( \PVD )=\HVD\circ \PVD\circ\TV\]  
\end{Definition}

The following commuting diagram illustrates the effect of $\WVD(\pd)$ on objects in the category $\ve$. 
\[
\begin{tikzcd}[column sep = 1.2in]
 W \arrow[r, "\TV"] \arrow[d, "\WVD(\pd)" left] & \wv  \arrow[d, "\pv"] \\
 \WVD(\pd)(W) &  \pv(\wv) \arrow[l, "\HVD"] 
\end{tikzcd} \]

From our definition of $\WVD (\cP_d)$, it is clear that this functor depends on the choice of the polynomial functor $\cP_d$ and the choice of a vector space $V$. Our goal is to be able to define a new functor, $\W_d:\polyd\to\polyd$. To be able to do so, we consider the following lemma. 

\begin{Lemma}
The functor $\WVD (\cP_d)$ on $\ve$ is a homogeneous polynomial functor of degree $d$.
\end{Lemma}

\begin{proof}
Recall the assumption that $P_d$ was itself a homogeneous polynomial functor of degree $d$. We have defined 
$\WVD(\cP_d) = \HVD \circ \PVD \circ \TV$
so that to establish the claim we need to analyze the three functors used here. First, notice that $\TV$ is a polynomial functor, being in particular a homogeneous linear functor. Moreover, we are given that $\pd$ is a homogeneous polynomial functor of degree $d$ and the induced functor $\PVD$ still retains this property. Finally, $\HVD$ is a homogeneous linear functor being the restriction to the $\glv$-invariant component of the homogeneous linear functor $\HVD = \Hom (\we^d V,$ \----). Thus, the composition of these three functor is a homogeneous polynomial functor of overall degree $d$.

\end{proof}

The above lemma allows us to define for every $\cP_d \in \poly_p$ a new object in $\polyd$, namely $\WVD(\PD)$. Notice that since we defined $\WVD$ as a composition of functors, its effect on morphisms in $\polyd$ (which are natural transformations between polynomial functors) is just the composition of the 
functors $(\HVD)_\star$ and $\TV^\star$.

\subsection{The functor $\WVD$ on Schur functors}

To understand the effect of the functor $\WN$ on $\poly$, we will first study the polynomial functor $\WN(\Sl)$ in $\poly$, for $\Sl$ the Schur functor associated to $\lambda$, a partition of $d$. 

\begin{Lemma}
Let $\lambda$ be a partition of $n$ and $d\geq \dim V$. The polynomial functor $\WN (\Sl)$ is naturally equivalent to $\Slp$.
\end{Lemma}

\begin{proof}
We have already showed that $\WN (\Sl)$ is a homogeneous polynomial functor of degree $d$. Notice that the functor $ \Sl \TV  = \Sl (\ten)$ can be decomposed using  the following formula

 \[ \Sl ( \ten) = \bigoplus ( \cS_{\mu} ( \blank ) \otimes \cS_{\nu} (V))^{a_{\lambda,\mu,\nu}} = \ldots \oplus \Slp(\blank) \otimes \we^d (V) \oplus \ldots ,\]
where $a_{\lambda,\mu,\nu}$ is the Kronecker coefficient (the tensor product multiplicity for the corresponding representations of the symmetric group).
We notice that in this decomposition the isotypic component of $\we^d(V)$ is given by $\Slp(\blank) \otimes \we^d (V)$
corresponding to the Kronecker coefficient $a_{\lambda,\lambda',(1^d)}=1$.
Consider now the effect of the functor $\HVD$. Only the image of the isotypic component of $\we^d (V)$ will be non-zero. In particular, 
\[ \HVD (\Slp(\blank) \otimes \we^n (V)) \cong( \we^n (V)^* \otimes \Slp(\blank) \otimes \we^n (V))^{\GL(V)} \cong \Slp(\blank), \]
where all the isomorphisms are natural equivalences. 
\end{proof}

Notice that in the above proof, we studied the image of $\WVD (\Sl)$ by examining the isotypic component of $\we^d (V)$ in $\Sl(\ten)$. When we apply the functor $\WN$, we will often use this computational approach to understand its effect on polynomial functors. In particular, one can use this approach to show that $\WN$ behaves well with respect to direct sums. The proof of the following lemma is left to the reader.
\begin{Lemma}
For polynomial functors $\cP_d$ and $\cP_d$ the functors $\WN(\cP_d\oplus \cP_d')$ and $\WN(\cP_d)\oplus \WN(\cP_d')$ are naturally equivalent.
\end{Lemma}

Moreover, we have that $\WN$ behaves well with respect to tensor products.
\begin{Lemma}
Let $\lambda$ and $\mu$ be partitions of $d$ and $e$ respectively. We have that $\Omega_{V,d+e} (\Sm \otimes \cS_{\nu})$ is naturally equivalent to $\Omega_{V,d} (\Sm) \otimes \Omega_{V,e} (\cS_{\nu}) $ if $\dim V\geq d+e$.
\end{Lemma}

\begin{proof}
Recall that an application of the Littlewood Richardson rule gives that:  
\begin{align*}
\Sm  \otimes  \cS_{\nu}  & =  \bigoplus_\lambda \Sl^{\clw}  \\
& = \bigoplus_\lambda \Sl^{c_{\mu' \nu'}^{\lambda'}} ,
\end{align*}
by the properties of the Littlewood Richardson coefficients. Applying $\Omega_{V,d+e}$ to this equations, we get that 
\[ \Omega_{V,d+e}(\cS_\mu  \otimes  \cS_{\nu} ) \cong  \bigoplus_\lambda \Omega_{V,d+e}(\Sl)^{c_{\mu' \nu'}^{\lambda'}}  \cong \bigoplus_\lambda \Slp^{c_{\mu' \nu'}^{\lambda'}}. \] 
On the other hand, we have that 
\begin{align*}
\Omega_{V,d} (\Sm) \otimes \Omega_{V,e} (\cS_{\nu}) \cong \Smp  \otimes  \cS_{\nu'}  & =  \bigoplus_\lambda \Slp^ {c_{\mu' \nu'}^{\lambda'}} .
\end{align*} 
 As $\Omega_{d+e} (\Sm \otimes \cS_{\nu})$ and $\Omega_d (\Sm) \otimes \Omega_e (\cS_{\nu}) $  have the same direct sum decomposition in terms of Schur functors, they are naturally equivalent polynomial functors.

\end{proof}

\subsection{The functor $\WV$ on $\poly$}
 
So far we have seen the effect of $\WN$ on Schur functors, on direct sums, and on tensor products. Using the graded structure of $\poly$, we can define $\WV$.
\begin{Definition}
Let $ \cP$ be an object in the category $\poly$. We can decompose $\cP$ in its graded pieces i.e., $ \cP = \bigoplus \pd$, where $\pd$ is a homogeneous polynomial functor of degree $d$. We define 
\[ \WV(\cP) = \bigoplus_d \WN (\pd) \]
\end{Definition}

However, for any homogeneous polynomial functor of degree $d$, we can choose a natural equivalence so that
\[ \pd \cong \bigoplus_{\lambda \dashv d} \cS_{\lambda}^{m_{\lambda}}, \]
for some integers $m_{\lambda}$. Then, using the previous results, we obtain that 
\[ \WN(\pd) \cong \WN(\bigoplus \cS_{\lambda}^{m_{\lambda}}) \cong \bigoplus \cS_{\lambda'}^{m_{\lambda}}. \]
In fact, we will actually often just think of the functor $\WN$ on $\poly_d$ in terms of its effect on Schur functors. Moreover, we use this point of view to compute the effect of $\WV$ on $\cP \in \Obj(\poly)$:
\[ \WV (\cP) = \bigoplus_d \WN (\pd) \cong \bigoplus_d \bigoplus_{\lambda \dashv d} \cS_{\lambda'}^{m_{\lambda}} .\]

Finally, we want to show that our definition is independent of the choice of $V$ i.e., if $V'$ is another vector space of dimension $m \geq d$, then for any $\pd \in \poly_d$ the functors $\WN(\pd)$ and $\W_{V',d}(\pd)$ are naturally equivalent.
\begin{Lemma}
Let $V,V'$ be two vector spaces of dimensions $n,m$, respectively, such that $n,m \geq d$. For every $\pd \in \poly_d$ we have that $\WN(\pd)$ and $\WVPD(\pd)$ are naturally equivalent functors. 
\end{Lemma}

\begin{proof}
Let $\pd \in \poly_d$. Then there exists a natural equivalence $\psi : \pd \to \bigoplus \Sl^{\ml}$.
By our definition of $\WN,\WVPD$  on $\Hom(\poly_d)$, we have that $\WN(\psi), \WVPD(\psi)$ are natural equivalences in $\Hom(\poly_d)$. Moreover, recall that $\WN(\bigoplus \Sl^{\ml})\cong \bigoplus \Slp^{\ml}$, by our results on the effect of $\WN$ on Schur functors. Let us call this natural equivalence $\phi$. Similarly, there exist a natural equivalence $\phi': \WVPD(\bigoplus \Sl^{\ml}) \to \bigoplus \Slp^{\ml}$. We will define $\eta_{\pd} : \WN(\pd) \to \WVPD(\pd)$ to be 
\[ \eta_{\pd} = \WVPD(\psi)^{-1} \circ \phi'^{-1} \circ \Id \circ \phi \circ \WN(\psi), \] 
or the top horizontal map in the following commuting diagram: 
\[
\begin{tikzcd} [column sep = 1 in, row sep = large]
 \WN(\pd) \arrow[r, "\eta_{\pd}"] \arrow[d, "\WN(\psi)"] & \WNp(\pd)  \\
 \WN(\bigoplus \Sl^{\ml}) \arrow[d, "\phi"] &  \WNp(\bigoplus \Sl^{\ml}) \arrow[u, "\WNp(\psi)^{-1}" ] \\
  \bigoplus \Slp^{\ml} \arrow[r, "\Id"]&   \bigoplus \Slp^{\ml} \arrow[u, "\phi'^{-1}"]
\end{tikzcd}  \quad .\]
In conclusion, notice that since $\eta_{\pd}$ is a composition of natural equivalences in $\Hom(\poly_d)$, it is itself a natural equivalence in $\Hom(\poly_d)$.
\end{proof}

As two vector spaces of dimension greater than $d$ yield naturally equivalent functors $\WN(\pd)$ and $\WNp(\pd)$ in $\poly_d$, for any functor $\cP$ in $\poly$ of degree $d$ (not necessarily homogeneous), we have that choosing vector spaces $V,V'$ of dimension greater than $d$ will yield the naturally equivalent functors $\WV(\cP)$ and $\WVp(\cP)$.

\subsection{The functor $\WV$ and the tensor structure of $\poly$}

Consider two Schur functors $\Sm, \Sn$, where $\mu , \nu$ are partitions of $d$ and $e$, respectively. Using the results from the previous section, we get that: 
\[ \WV(\Sm \otimes \Sn ) = \W_{V, d+e} ( \Sm \otimes \Sn) \cong \W_{V,d}( \Sm) \otimes \W_{V,e}(\Sn) \cong \WV(\Sm) \otimes \WV(\Sn)\]
where $V$ is a vector space of dimension greater or equal to $d+e$.

However, we have not explicitly produced a natural equivalence between these functor. In particular, we have not studied how the functor $\WV$ interacts with the symmetric tensor structure of $\poly$. Recall that for every $\cP, \cP' \in \poly$, there are natural equivalence $s_{\cP,\cP'}, s_{\cP',\cP}$, where 
\[
\begin{tikzcd} [column sep = large, row sep = large]
 \cP \otimes \cP' \arrow[r, "s_{\cP,\cP'}"] & \cP' \otimes \cP \arrow[r, "s_{\cP',\cP}"] & \cP \otimes \cP'
\end{tikzcd}  \quad ,\]
such that $s_{\cP',\cP} \circ s_{\cP,\cP'} = \Id_{ \cP \otimes \cP'}$.
Similarly, there are natural equivalences $s_{\WV(\cP),\WV(\cP')}$, $s_{\WV(\cP'),\WV(\cP)}$ such that 
\[ s_{\WV(\cP'),\WV(\cP)}\circ s_{\WV(\cP),\WV(\cP')} = \Id_{ \WV(\cP) \otimes \WV(\cP')}. \]

The question arising from this set up is whether we can produce a natural equivalence $\psi_{\cP,\cP'}: \WV(\cP) \otimes \WV(\cP')  \to \WV(\cP \otimes \cP') $ compatible with the tensor structure. In practice, we want to determine if the following diagram commutes:
\[
\begin{tikzcd} [column sep = 1 in, row sep = large]
 \WV( \cP \otimes \cP') \arrow[r, "\WV (s_{\cP,\cP'})"] &   \WV( \cP' \otimes \cP)  \\
\WV(\cP) \otimes \WV(\cP') \arrow[u, "\psi_{\cP,\cP'}"] \arrow[r, "s_{\WV(\cP),\WV(\cP')}"] &  \WV(\cP') \otimes \WV(\cP) \arrow[u, "\psi_{\cP',\cP}"]
\end{tikzcd}  \quad .\]

As every polynomial functor is a direct sum of Schur functors, it will be enough to study this diagram for $\cP=\Sl$ and $\cP'= \Sm$, for $\lambda$ a partition of $d$ and $\mu$ a partition of $e$.

\begin{Proposition}
We can define $\psi_{\Sl,\Sm}$ so that the above diagram commutes up to sign $(-1)^{de}$.
\end{Proposition}

\begin{proof}
First we will define $\psi_{\Sl,\Sm}$ and from our definition we will conclude that the diagram only commutes up to a sign.

As remarked before, $\WVD (\Sl) (W) $ is the multiplicity space of the isotypic component of the irreducible $\glv$-representation $\we^d(V)$ inside $\Sl(\wv)$. Thus, to find a natural equivalence from $\WVD(\Sl)\otimes \WVE(\Sm)$ to $\Omega_{V,d+e}(\Sl \otimes \Sm)$  we need to exhibit for any finite dimensional vector space $W$ an isomorphism from the tensor product of the isotypic component of $\we^d(V)$ in $\WVD(\Sl)(W)$  and the isotypic component of  $\we^e(V)$ in  $\WVE(\Sm)(W)$ to the isotypic component of $\we^{d+e}(V)$ in $\WV(\Sl \otimes \Sm)$.

Going back to the definition of $\WN (\pd)(W)$, we notice that this is equivalent to producing an isomorphism $\phi=\psi_{\Sl,\Sm}(W)$:
\[
\begin{tikzcd} [column sep = 1 in, row sep = large]
\Hom_{\glv} (\we^d (V), \Sl(\wv)) \otimes \Hom_{\glv} (\we^e (V), \Sm(\wv))  \arrow[d, "\phi"] \\
\Hom_{\glv} (\we^{d+e} (V), \Sl(\wv) \otimes \Sm (\wv))
\end{tikzcd}  \quad \]
for every $W$. 

Note that $\bigwedge^d(V)\otimes \bigwedge^e(V)$ has a unique sub-representation isomorphic to $\bigwedge^{d+e}(V)$.
Now we define $\phi$ as follows. If $f:\bigwedge^d(V)\to \Sl(\wv)$ and $g:\bigwedge^e(V)\to \Sm(\wv)$ are $\GL(V)$-equivariant linear maps, 
then we define $\phi(f\otimes g)$ as the restriction of $f\otimes g:\bigwedge^d(V)\otimes \bigwedge^e(V)\to \Sl(\wv)\otimes \Sm(\wv)$
to the sub-representation $\bigwedge^{d+e}(V)\subseteq \bigwedge^d V\otimes \bigwedge^e V$. We can extend $\phi$ to a linear map.
 
By Schur's lemma any $\glv$-equivariant map in $f: \we^d(V) \to \Sl(\wv)$  can be written as 
$$f = \Id \otimes a:\we^d(V)\to \we^d(V)\otimes \Slp(W)\subseteq \Sl(\wv),$$ 
for $a$ some constant map from $\we^d(V)$ to $\Slp(W)$, as $\Slp(W)$ is the multiplicity space of the isotypic component of $\we^d(V)$ in $\Sl(\wv)$. Similarly, letting  $b$ be a constant map to $\Smp(W)$, we will have that $g \in  \Hom_{\glv} (\we^e (V), \Sm(\wv))$ can be written as $g= \Id \otimes b$. As the multiplicity space of $\we^{d+e} (V)$ in  $\Sl(\wv) \otimes \Sm (\wv)$ is precisely $\Slp(W) \otimes \Smp (W)$, we have that $\phi$  sends
 $\Sigma_{i} \Id \otimes a_i \otimes \Id \otimes b_i$ to $ \Sigma_{i} \Id \otimes a_i \otimes b_i$. As this map is an injective map between isomorphic spaces, it is an isomorphism.

In the diagram below:
\[
\begin{tikzcd} [column sep = 1 in, row sep = large]
\we^d (V) \otimes \we^e (V)  \arrow[d, "s_{\we^d (V),\we^e (V)}"] \arrow[r, hookleftarrow] & \we^{d+e} (V)   \\
\we^e (V) \otimes \we^d (V)  \arrow[ur, hookleftarrow] &
\end{tikzcd}  \quad .\]  
the map $s_{\we^d (V),\we^e (V)}$ takes the pure tensor $a\otimes b$ to $(-1)^{de}b \otimes a$. Thus, the diagram only commutes up to sign $(-1)^{de}$.  

So applying $\Hom ( \Blank, \Sl \otimes \Sm (\wv))$ to the diagram above, we obtain a new diagram that only commutes up to sign $(-1)^{de}$.
\[
\begin{tikzcd} 
\Hom( \we^d (V) \otimes \we^e (V), \Sl \otimes \Sm (\wv))  \arrow[r] & \Hom(\we^{d+e} (V), \Sl \otimes \Sm)(W) \\
\Hom( \we^e (V) \otimes \we^d (V), \Sl \otimes \Sm (\wv)) \arrow[u] \arrow[ur] & 
\end{tikzcd}  \quad .\] 
Notice that restricting the spaces above to the relevant invariant subspaces does not affect the sign in the above diagram. 

The map
\[ \WV (s_{\Sm,\Sl})(W) : \WV ( \Sm \otimes \Sl) (W) \to \WV (\Sl \otimes \Sm) (W) \]
is just given by $ \Sigma_{i} \Id \otimes b_i \otimes a_i \mapsto \Sigma_{i} \Id \otimes a_i \otimes b_i$, so that on pure tensors we have that $\Id \otimes b \otimes a \mapsto \Id \otimes a \otimes b$. Finally, consider the diagram below
\[
\begin{tikzcd} [column sep = 1 in, row sep = large]
\WV(\Sl)(W) \otimes \WV(\Sm)(W) \arrow[r, "\phi"] & \WV (\Sl \otimes \Sm)(W) \\
\WV(\Sm)(W) \otimes \WV(\Sl)(W)  \arrow[u,  "s_{\WV(\Sm)(W),\WV(\Sl)(W)}"] \arrow[r, "\phi"] & \WV( \Sm \otimes \Sl)(W) \arrow[u,  "\WV (s_{\Sl,\Sm})(W)"] 
\end{tikzcd}  \quad. \]
We can conclude that the diagram only commutes up to sign $(-1)^{de}$ as first going right and then up sends $\Id \otimes b \otimes \Id \otimes a \mapsto \Id \otimes a \otimes b$, whilst first going up then right send $\Id \otimes b \otimes \Id \otimes a \mapsto (-1)^{de} \Id \otimes a \otimes b$

\end{proof}

\subsection{Definition of $\W(\cP)$}

After examining the effect of $\WN$ on the tensor structure of $\poly$, we want to be able to work with a functor defined independently from the choice of the vector space $V$. For every $i\geq 0$, define $V_i\cong \KK^i$ as the vector space of all sequences $(a_1,a_2,a_3,\dots)\in \KK^\infty$ with
$a_j=0$ for all $j>i$. Let $\rho_{ji}$ be the inclusion $\rho_{ji}: V_i\to V_j$.

Suppose that $i\leq j$. Consider $\Omega_{d,V_j}(\cP_d)$, by definition we have that $\Omega_{d,V_j}(\cP_d)\in \Hom_{\GL(V_j)}(\bigwedge^d V_j,\cP_d(V_j\otimes \blank))$. 
The restriction of $\Omega_{d,V_j}(\cP_d)$ to its subspace $\bigwedge^d V_i$ is $\GL(V_i)$-equivariant and the image is contained in $\cP_d(V_i\otimes \blank)$.
So we have a natural transformation $g_{i,j}:\Omega_{d,V_j}(\cP_d)\to \Omega_{d,V_i}(\cP_d)$.
The kernel of this natural transformation consists exactly of all isotypic components $\cS_{\lambda'}$ where $\lambda$ has more than $i$ parts.
There is a unique splitting $f_{j,i}: \Omega_{d,V_i}(\cP_d)\to \Omega_{d,V_j}(\cP_d)$ such that $g_{i,j}\circ f_{j,i}$ is the identity.
We have a direct system
$$
\xymatrix{ \Omega_{d,V_0}\ar[r]^{f_{0}} &  \Omega_{d,V_1}\ar[r]^{f_{1}} &  \Omega_{d,V_2}\ar[r]^{f_{2}} & \cdots}
$$
where $f_i=f_{i+1,i}$.

\begin{Definition}
 We define $\W_d(\cP_d)$ to be the following direct limit of the maps $f_{j,i}$: 
\[ \W_d(\cP_d) = \varinjlim \W_{d,V_i}(\cP_d) .\]  
\end{Definition}   
Notice that the abelian category of polynomial functor is cocomplete so this functor is well-defined as direct limits exist. Finally, recall our definition of $\WV$ for $\cP \in \poly$ decomposed as $\cP = \bigoplus_d \cP_d$:
\[\W_V (\cP) = \bigoplus_d \W_{V,d} ( \cP_d). \]

Similarly, we have the following definition for $\W$.
\begin{Definition}
Let $\cP \in \poly$ be decomposed as $\cP = \bigoplus_d \cP_d$. We define $\W(\cP)$ to be: 
\[ \W(\cP) = \bigoplus_d \W_d ( \cP_d) .\]  
\end{Definition}

We can notice that our definition of $\WV$ is compatible with our definition of $\W$:
\[\W(\cP) = \bigoplus_d \W_d ( \cP_d) = \bigoplus_d \varinjlim \W_{d,V_i}(\cP_d) = \varinjlim  \bigoplus_d \W_{d,V_i}(\cP_d) = \varinjlim \W_{V_i}(\cP),\]
as direct sums and direct limits commute.

\subsection{The functor $\W$ on $\gp$}

Recall that the category $\gp$ is a functor category where each functor $\cF \in \gp$ can be decomposed as 
\[ \cF = \bigoplus \cF_d \]
with $\cF_d$ a homogeneous polynomial functor of degree $d$. Moreover, recall that each polynomial functor is naturally equivalent to a direct sum of $\Sl$'s. Thus we can also notice that 
\[ \cF \cong \bigoplus \Sl^{\ml} , \]
for $\lambda$'s of any length, but with the stipulation that for each polynomial functor $\cF_d$ and any $W \in \Obj(\ve)$, we have that $\cF_d (W) \cong \bigoplus_{\lambda \dashv d} \Sl^{\ml} (W)$ is a finite dimensional vector space. 

In the previous section we have seen how to define $\W$ on any polynomial functor. We can extend this definition to any direct sum of polynomial functors, even though the direct sum itself may not be a polynomial functor. Thus, we can define the functor $\W$ on the category $\gp$. In particular, we have that $\cS$, the symmetric algebra functor, is not a polynomial functor because for any $W \in \Obj(\ve)$  we have that $\cS(W)$, the symmetric algebra on $W$, is an infinite dimensional vector space. However, $\cS$ is a direct sum of polynomial functors as each graded piece $\cS^d(W)$ is a finite dimensional vector space. Thus $\cS$ is an object in $\gp$. 

\begin{Definition}  
For any functor $\cF \in \Obj( \gp)$, where $\cF= \bigoplus \cF_d$, we let
\[ \W(\cF) = \bigoplus \W_d (\cF_d) . \] 
\end{Definition}

In particular, for $\cS = \bigoplus \cS_(d) $, we have that 
\[\W (\cS) = \bigoplus \W_d(\cS_(d)) \cong \bigoplus \we^d = \we \] 
from our results on the effect of $\W$ on Schur functors.

\section{Resolutions in $\gp$}

\subsection{Equivariant resolutions}

Let us fix a vector space $U$ of dimension $n$. The ring of polynomial functions on $U$ can be identified with the symmetric algebra $S(U^*)$. Let $R = S(U^*)$ and let $\mathfrak{m}$ be the homogeneous maximal ideal in $R$. Given a module $M$ over $R$ we can construct a minimal resolution by defining $D_0 := M$ and $E_0 := D_0/ \mathfrak{m}D_0$. We can then extend in a unique way the homogeneous section $\phi_0:E_0\to D_0$ of the homogeneous quotient map $\pi_0:D_0\to E_0$ to a $R$-module homomorphism $\phi_0: R \otimes E_0 \to D_0$. The tensor product $R \otimes E_0$ is naturally graded as a tensor product of graded vector spaces and $\phi_0$ is homogeneous with respect to this grading.  Letting $D_1$ be the kernel of $\phi_0$, we see how to proceed inductively to construct a free resolution of $M$. The resolution is finite by Hilbert's syzygy theorem which states that $D_i=0$ for $i>n$. In the resulting minimal free resolution:
\[ 0 \to R \otimes E_t \to \cdots \to R \otimes E_0 \to M \to 0, \]
we can naturally identify $E_i$ with $\Tor_i(M, \KK)$. 

Moreover, suppose that $U$ is a representation of a linearly reductive algebraic group $G$. If $G$ also acts on the module $M$ and the multiplication map $m: R \times M \to M$ is $G$-equivariant, then each graded piece $M_d$ of $M= \bigoplus M_d$ is a $G$-module. By linear reductivity, we can choose the maps $\phi_i$ to be $G$-equivariant giving each $E_i$ the structure of a graded $G$-module. In particular, we can choose a decomposition of each $E_i$ into irreducible $G$-representations. In the case of $G = \glv$ and $M$ a polynomial representation of $G$, we will have a decomposition of each $E_i$ in the equivariant resolution of $M$ in terms of the Schur functors $\Sl$'s. As a result we have the following equivariant resolution of the $G$-module $M$:
\[ 0 \to R \otimes \bigoplus \Sl(V)^{m^n_{\lambda}} \to \cdots \to  R \otimes \bigoplus \Sl(V)^{m^0_{\lambda}} \to M \to 0. \]  
In particular, we have that $M/ \mathfrak{m}M = \Tor_0(M, \KK)$ can be decomposed in irreducible $\glv$-representations as $\bigoplus \Sl(V)^{m^0_{\lambda}}$.

\subsection{The functor $\boldsymbol{\W}$ and resolutions in $\gp$}
Consider a module functor $\cM$ in $\gp$ over the algebra functor $\cR$ in $\gp$. A resolution for $\cM$ is constructed analogously to a resolution for a $\glv$-module.  In particular, the construction relies on the fact that $\gp$ is a semisimple category as each object is naturally equivalent to a direct sum of simple objects: the irreducible polynomial functors $\Sl$'s. 

The object $\cM$  in $\gp$ is equipped with a natural equivalence
\[ \cM \cong \bigoplus \Sl \otimes A_{\lambda}, \]
where $A_{\lambda}$ is the multiplicity space of $\Sl$, a vector space recording the multiplicity of the polynomial functor $\Sl$.

Let $\cN$ be another $\cR$-module functor, where $\cN \cong \bigoplus \Sl \otimes B_{\lambda}$. We have that the map $\psi: \cM \to \cN$ can be viewed as a map 
\[ \bigoplus \Sl \otimes A_{\lambda} \to  \bigoplus \Sl \otimes B_{\lambda} \]
and being a homomorphism in $\gp$, we have that $\psi = \bigoplus Id \otimes \psi_{\lambda}$, where each $\psi_{\lambda}: A_\lambda \to B_\lambda$ is just a linear map. Thus each $\psi_\lambda$ has a section $\phi_\lambda : B_\lambda \to A_\lambda$ allowing us to construct $\phi := \bigoplus Id \otimes \phi_{\lambda}:\cN\to\cM$, a section of $\psi$. 

As $\gp$ is an abelian category for each $\phi \in \Hom( \gp)$, there exists an object $\cK \in \Obj( \gp)$ which is the kernel of $\phi: \cN \to \cM$.

To construct the minimal resolution of $\cM$, we will consider the map $\psi_0 : (\cM/ \mm \cM) \otimes \cR \to \cM$, where $\mm$ is the positively graded part of $\cR \in \gp$, the maximal homogeneous module functor of $\cR$. 
We will define $\cD_0:=\cM$ and $\cE_0 := \cM/ \mm \cM $. Then $\cD_1:= \ker (\phi_0)$, where $\phi_0$ is the $\cR$-module functor morphism arising from a section of $\psi_0$ as discussed above. Inductively, we define $\cD_i = \ker \phi_{i-1}$ and $\cE_i = \cD_i / \mm \cD_i$. Then we will let $\phi_i : \cE_i \to \cD_i$ be the section of the quotient map $\psi_i : \cD_i \to \cE_i$ and we will extend $\phi_i$ to a $\cR$-module functors map $\phi_i : \cR \otimes \cE_i \to \cD_i$. As a result we obtain the following minimal resolution for $\cM$:
\[ \cdots \to \cR \otimes \bigoplus \Sl^{m^i_{\lambda}} \to \cdots \to \cR \otimes \bigoplus \Sl^{m^0_{\lambda}} \to \cM \to 0, \] 
where $\cE_i \cong \bigoplus \Sl^{m^i_{\lambda}}$.

Notice that for $W \in \Obj(\ve)$ and $\cR= \cS(W)$, a minimal resolution for $\cM(W)$ will terminate as by Hilbert's syzygy theorem $\cD_i = 0$ when $i > \dim(W)$.  However, the categorical construction of the resolution of $\cM$ may be infinite.

Given a minimal resolution of $\cM$ constructed as above, we can apply $\W$ to the resolution to obtain
\[ \cdots \to \W(\cR) \otimes \bigoplus \Slp^{m^i_{\lambda}} \to \cdots \to \W(\cR) \otimes \bigoplus \Slp^{m^0_{\lambda}} \to \W(\cM) \to 0.\]
In particular, notice that if $\cR=\cS$, the symmetric algebra functor, then applying $\W$ to a resolution of the $\cS$-module functor $\cM$ will result in a resolution of $\W(\cM)$, a module over the algebra functor $\W(\cR) \cong \we$. 
 
\section{Castelnuovo-Mumford regularity of modules in $\gp$}
For a finite dimensional graded $\KK$-vector space $E = \bigoplus E_d$, we define
\[ \deg(E):= \max \{ d : E_d \neq 0\}.\]
If $E=\{0\}$, then we define $\deg(E)= - \infty$. For $\cM, \cR$ in $\Obj(\gp)$, let $\cM$ be a module functor over the algebra functor $\cR$. For every $V$ in $\Obj(\ve)$ we have that $\cM(V)$ is a module, in the usual sense, over the $\KK$-algebra $\cR(V)$. We have that $\cM(V)$ is $s$-regular if $\deg(\Tor_i(M(V), \KK)) \leq s+i$, for all $i$. In particular, notice that using the minimal resolution constructed above, we have that $ \Tor_i(\cM(V), \KK)) = \cE_i(V)$. Thus, we have that $\cM(V)$ is $s$-regular if
\[ \deg(\cE_i(V)) \leq s+i, \] 
for all $i$. 

The \emph{Catelnuovo-Mumford regularity} $\reg(\cM(V))$ of $\cM(V)$ is the smallest integer $s$ such that $\cM(V)$ is $s$-regular. We define the \emph{regularity} of $\cM \in \Obj(\gp)$ to be 
\[ \lim_{\dim(V)\to \infty} \reg(\cM(V)) \]
if the limit exists. 

\begin{Proposition}\label{keyprop}
Let $\cM$ be a module over $\cR$ in $\gp$ with regularity $d$. Then $\W(\cM)$ is a module over $\W(\cR)$ with regularity $d$.
\end{Proposition}

\begin{proof}
As the module $\cM$ in $\gp$ comes equipped with a multiplication map $\nu: \cR \otimes \cM \to \cM$, we have that $\W(\nu): \W(\cR) \otimes \W(\cM) \to \W(\cM)$ equips $\W(\cM)$ with the structure of a $\W(\cR)$-module. Consider a minimal resolution for $\cM$ as constructed above:
\[ \cdots \to \cR \otimes \cE_i \to \cdots \to \cR \otimes \cE_0 \to \cM \to 0.\]
We can apply $\W$ to the resolution to obtain
\[ \cdots \to \W(\cR) \otimes \W(\cE_i) \to \cdots \to \W(\cR) \otimes \W(\cE_0) \to \W(\cM) \to 0.\]

Notice that for every vector space $V$, we have that $\deg(\W(\cE_i)(V)) \leq \deg(\cE_i(V))$. In fact, it is possible that $\cE_i(V) \neq 0$ but $\W(E_i)(V) = 0$ for $\dim V \leq i$. However, for $\dim (V)$ large enough, we have that $\deg(\W(\cE_i)(V)) = \deg \cE_i(V)$. Hence for $\dim(V)$ large enough, we also have that $\W(\cE_i)(V) = \Tor_i(\W(\cM)(V), \KK)$. We have that $\cM(V)$ is $s$-regular if $\max_i \{\deg(\cE_i(V))-i\} \leq s$. Thus, for $\dim(V)$ large enough, we have that 
\[ \max_i \{\deg(\W(\cE_i)(V))-i\} = \max_i \{\deg(\cE_i(V))-i\} \leq s,\]
so $\W(\cM)(V)$ is $s$-regular whenever $\cM(V)$ is $s$-regular.

Therefore,
\[ d = \reg(\cM) = \lim_{\dim(V)\to \infty} \reg(\cM(V)) = \lim_{\dim(V)\to \infty} \reg(\W(\cM)(V))= \reg \W(\cM).\]   
\end{proof}

\section{The module functors of a subspace arrangement}
For $Y$ a $\KK$-vector space, a subspace arrangement $\cA = \{ Y_1, \dots , Y_t\}$ is  a collection of linear subspaces in $Y$. The ideal associated to $\cA$ is the vanishing ideal of the subspace arrangement: $I_{\cA} = \II (\cA) = \II (Y_1 \cup \cdots \cup Y_t )$. Moreover, we can define $J_{\cA} = \prod_{i} \II(Y_i) =  \II(Y_1)\II(Y_2)  \cdots \II(Y_t)$

Let $W = Y^*$. Then we have that $I_{\cA}, J_{\cA} $ are ideal in $\sym(Y^*) = \cS (W)$.
\begin{Definition}
Let $V$ be any object in $\ve$ and let $Z= V^*$. In the polynomial ring $\cS(W \otimes V)$ we define $\ia (V)$ to be the vanishing ideal of the subspace arrangement $\cA \otimes Z$, i.e.,
\[ \ia(V) = \II( Y_1\otimes Z \cup \cdots \cup Y_t \otimes Z). \]
Moreover, we define $J_{\cA}$ to be the product ideal
\[ \ja(V) = \II( Y_1\otimes Z) \II( Y_2\otimes Z) \cdots \II(Y_t \otimes Z). \]
\end{Definition}

For any subspace arrangement $\cA$, we can now construct the module functors $\ia, \ja$ in $\gp$ for the algebra functor $\cS(W \otimes \blank)$. For any object $V$ in $\ve$, we have already defined $\ia(V)$ and $\ja(V)$. Notice that for every vector space $V$, we have that $\ia(V), \ja(V)$ are homogeneous ideals in $\cS(\wv)$ so that can define monomorphisms $\ia, \ja \into \cS(W \otimes \blank)$. Thus for every $d$, we have that $(\ia)_d, (\ja)_d $ are polynomial functors of degree $d$ giving $\ia$ and $\ja$ the structure of objects in $\gp$.

To show that $\ia$ and $\ja$ are module functors in $\gp$, as they inherit a multiplication map from $\cS(W \otimes \blank)$, we only need to define $\ia, \ja$ on $\Hom (\ve)$. For ease of notation, will proceed with the definition for $\ia$, but the same construction works for $\ja$. In fact, the reader may substitute $\ja$ for $\ia$ in the following paragraphs without affecting the results. Let $f : V_1 \to V_2$ be in $\Hom ( \ve)$. Then also $Id \otimes f : W \otimes V_1 \to W \otimes V_2$ is in $\Hom (\ve)$, As $\cS( W \otimes \blank)$ is an object in $\gp$, we have $\cS( \Id \otimes f):  \cS( W \otimes V1) \to \cS( W \otimes V_2)$ in $\Hom (\gv)$.

\begin{Proposition}
For any $f: V_1 \to V_2$ in $\Hom(\ve)$, we have that $\ia (f)$ defined as $\cS( \Id \otimes f)|_{\ia(V_1)}$ is an element in $\Hom (\gv)$ such that 
\[ \cS( \Id \otimes f)|_{\ia(V_1)} : \ia(V_1) \to \ia(V_2). \]
Moreover, this shows that $\ia$ is a module functor in $\gp$.
\end{Proposition}

\begin{proof}
Notice that since $\ia(V_1)$ is an ideal in $\cS( W \otimes V_1)$, we can restrict $\cS( \Id \otimes f)$ to $\ia(V_1)$. Moreover, we have that in degree one $\cS( \Id \otimes f)_1= 1d \otimes f$ and that $\ia (V)$ is a linear ideal for any vector space $V $ meaning that $\ia(V)$ is generated in degree one. Thus, we only need to show that the linear generators of $\ia (V_1)$ get mapped by $\cS(\Id \otimes f)_1=\Id \otimes f$ to the linear generators of $\ia(V_2)$ and that $\cS(\Id \otimes f)$ is an algebra homomorphism, so that it maps ideals to ideals.

In general, for any $g \in \Hom (\ve)$ by our definition of an algebra functor $\cR$ in $\gp$ we have that $\cR(g)$ is an algebra homomorphism. In particular, $\cS(\Id \otimes f)$ is algebra homomorphism. 

With regards to the generators of $\ia(V_1)$, let $w \in I_{\cA}$, so that $w=0$ on $\cA$. Then for any $v \in V$ we have that $w \otimes v$ is a linear generator of $\ia (V)$. Moreover, if $w \otimes v_1 \in \ia (V_1)$ then $\ia(f)(w \otimes v_1) = w \otimes f(v_1)$ is a linear generator of $\ia(V_2)$ as $v_2 = f(v_1) \in V_2$ and $w \in I_{\cA}$.

Finally, as $\ia(f)$ is the restriction of the functorial map $\cS( \Id \otimes f)$, we have that $\ia(f)$ is itself functorial.
 
\end{proof}

Using the construction above, we can establish our main result.

\begin{Theorem}\label{reg}
For any subspace arrangement $\cA$ of size $t$ consider the module functor $\ia$ in $\gp$. For any vector space $V$, we have that $\W(\ia)(V)$ is a $t$-regular $\glv$-equivariant ideal in $\W(\cS(\wv))= \we(\wv)$.
\end{Theorem}

\begin{proof}
Derksen and Sidman proved in ~\cite{ds1} that the intersection of $t$ linear ideals is $t$-regular and this implies that $\ia(V)$ is $t$-regular. We have previously seen that $\W(\cS(\wv))= \we(\wv)$ and that the image under $\W$ of a module functor is a module functor. Furthermore, $\W$ is an exact functor on $\gp$ so that monomorphisms are sent under $\W$ to monomorphisms. Thus $\W(\ia)(V)$ is an ideal in $\we(\wv)$. Finally, we have already established that $\W$ preserves the regularity of a module functor. Therefore, for every vector space $V$, we have that $\W(\ia)(V)$ is $t$-regular. 
\end{proof}

Conca and Herzog showed in \cite{ch} that the product of $t$ linear ideals is $t$-regular and this implies that $\ja(V)$ is also $t$-regular for al $V$. Therefore, the same result holds for $\W(\ja)(V)$.  Moreover, Dersken and Sidman in \cite{ds2} produce regularity bounds for a more general class of ideals constructed from linear ideals in the symmetric algebra. One can adapt Theorem \ref{reg} to establish that the same class of ideals in the exterior algebra has the same regularity bounds.

Consider the module functor $\ja$ for the product ideal $J_{\cA}$. We can characterize the ideal $\W(\ja)(V)$ in the exterior algebra $\we(\wv)$.
\begin{Proposition}\label{mult}
Let $\cA$ be the a subspace arrangement of cardinality $t$ and let $J_i$ be the vanishing ideal of the $i$-th subspace in $\cA$. We have that for every finite dimensional vector space $V$
\[ \W(\ja)(V) = J_1 (V) \wedge J_2 (V) \wedge \cdots \wedge J_t(V). \]
\end{Proposition}

\begin{proof}
Notice that $\ja(V) = J_1(V) J_2(V) \cdots J_t(V)$, using the multiplication structure of the symmetric algebra $\cS (\wv)$. The functor $\Omega$ maps the multiplication map of $\cS$ to a multiplication map in $\W(\cS)$. Up to scalars, there is a unique $\glv$-equivariant multiplication in $\we$, namely the multiplication given by $\wedge$. 
\end{proof}

Using the proposition above, we can reformulate Theorem ~\ref{reg} as general statement in commutative algebra.

\begin{Theorem}\label{bignice}
		The wedge product of $t$ linear ideals in the exterior algebra is $t$-regular.
\end{Theorem}

\begin{proof}
Every linear ideal $J_i$ is a vanishing ideal of a subspace $W_i$. Consider the subspace arrangement $\cA$ given by a set of linear ideals. Applying Proposition ~\ref{keyprop} we conclude that the associated module functor $\W(\ja)(V)$ in $\we(\wv)$ is $t$-regular. Using Proposition \ref{mult} for $V$ a 1-dimensional vector space, we conclude that the wedge product of the linear ideals is $t$-regular.
\end{proof}

\section{Equivariant Hilbert series and examples}
Let $V$ be an $n$-dimensional vector space. We denote by $s_{\lambda}(x_1, \ldots, x_n)$ the symmetric function which is the character of the irreducible representation $\cS_{\lambda}(V)$.
To the polynomial functor $\cS_\lambda$ we associate the symmetric function $s_\lambda=s_{\lambda}(x_1,x_2,\dots)$ in infinitely many variables. The character of the symmetric algebra functor $\cS=\bigoplus_{d=0}^\infty \cS_d$ is $\sigma=1+s_1+s_2+\cdots$. We call this series of symmetric function the equivariant Hilbert series of $\cS(V)$. The equivariant Hilbert series of the functor $V\mapsto \cS(V\oplus V)=\cS(V)\otimes \cS(V)$ is $(1+s_1+s_2+s_3+\cdots)^2=
1+(2s_1)+(3s_2+s_{1,1})+\cdots$.
As a result of the properties of $\Omega$, we have that the character of $\Omega(\cS_{\lambda})(V)$ is $s_{\lambda'}(x_1, \ldots, x_n)$. For a polynomial functor $\cF$, we consider the symmetric function $H^e(\cF)$ such that $H^e(\cF)(V)$ is the character of $\cF(V)$ as a representation of $\glv$. We refer to $H^e(\cF)$ as the equivariant Hilbert series of the polynomial functor $\cF$. We state the following result for $\ja$, the product module functor of a subspace arrangement $\cA$, but the same result holds for $\ia$, the intersection module functor. 
  
\begin{Theorem}
Let $\cA$ be a subspace arrangement.  Consider $\cJ_{\cA}(V)$, its associated $\glv$-equivariant product ideal in the symmetric algebra,  and $\Omega(\cJ_{\cA})(V)$, the ideal in the exterior algebra obtained by applying $\W$ to $\ja$. We have that
\[ H^e(\Omega(\cJ_{\cA})(V)) = \omega (H^e(\cJ_{\cA}(V))), \]
where $\omega$ is the involution on the ring of symmetric functions sending $s_{\lambda}$ to $s_{\lambda'}$.
\end{Theorem}

\begin{proof}
Consider an equivariant resolution of $\ja$. As $\ja$ has a linear resolution, we can read off $H^e(\ja)$ from the resolution (as discussed by Derksen in ~\cite{sym}). Apply $\W$ to the resolution. Each Schur functor $\Sl$ is mapped by $\W$ to $\Slp$. The effect on $H^e(\ja)$ is to change each $s_{\lambda}$ to $s_{\lambda'}$. Thus the new equivariant Hilbert series is $\omega (H^e(\ja))$. However we now have a resolution of $\W(\ja)$, so $\omega (H^e(\ja))$ is its equivariant Hilbert series.
\end{proof}

Consequently, an equivariant Hilbert series of $\Omega(\cJ_{\cA})$ can immediately be obtained from an equivariant Hilbert series of $J_{\cA}$. As we have a recursive combinatorial for $H^e(\ja)$ from \cite{sym}, we can find write down a resolution for $\ja$ and obtain a resolution for $\W(\ja)$ from its equivariant Hilbert series $H^e(\Omega(\cJ_{\cA})=\omega (H^e(\ja))$.

\subsection{Computing equivariant Hilbert series via polymatroids}
Consider a subspace arrangement $\cA = \{ W_1, \ldots, W_t\} \subset W$. Let $A = \{1, 2, \ldots, t\}$ be the indexing set of the subspaces in $\cA$. Each choice of subset $B \subset A$ gives us the subarrangement $\cB = \{ W_i | i \in B\}$ so that any subset of indexes $B$ gives us a product ideal
\[ J_B = \prod_{i \in B} J_i, \]
where $J_i = \II(W_i)$. Additionally, we can define a map $\phi_{B,C}$ between a subset $B$ of size $s$ and a subset $C$ of size $s-1$. We have that $\phi_{B,C} \neq 0 $ only if $C \subset B$. If $B = \{ a_1, \dots, a_s\}$ and $C = \{ a_1, \dots,a_{i-1}, a_{i+1}, \dots,  a_s\}$, then $\phi_{B,C} = (-1)^i \Id$.

These maps allow us to construct the following complex:
\[\mathcal{C}: 0 \to J_A \to \bigoplus_{|B|=t-1} J_B \to \cdots \bigoplus_{|B|=2} J_B \to J_1 \oplus J_2 \oplus \cdots \oplus J_t \to R \to 0\]
where, for example, the first map is the direct sum of the maps 
\[ \phi_{A, \{1,2, \ldots, i-1,i+1, \ldots, t\} } = (-1)^i \Id. \]

One can check that this is indeed a complex: the composition of two consecutive maps is zero. One can consider the homology of the complex, given by $\frac{\ker \phi}{\im \phi}$. Because our ideals are linear, a result of Conca and Herzog \cite{ch} tells us that the homology of the product complex is well-behaved. If the intersection of the subspaces is $0$, then the homogeneous maximal ideal $\mathfrak{m}$ kills the homology of this complex and they show that the $k$th homology is concentrated in degree $k$. In fact, this fact is key in showing that the product of $t$ linear ideals has regularity $t$.

For any $V \in \ve$, when we apply the tensor trick to $\cA$ to obtain $\cA\otimes V$. As a result, we just tensor every term in the complex $\cC$ with $V$ to obtain the complex $\cC \otimes V$.  
Even though the complex is not exact, we have that up to low degree terms 
\[ 0 \approx \sum_{B \subseteq A} (-1)^{|B|} H^e(J_{\cB}(V)), \]
so that for any $V$ we have that
\[(-1)^{n+1} H^e(J_{\cA}(V)) \approx \Sigma_{B \subset A} (-1)^{|B|} H^e(J_{\cB}(V)).\]

In particular, we can use this approach to compute some equivariant Hilbert series inductively. Specifically, we have the following result.

\begin{Proposition}\label{lines}
Let $\cA$ be the subspace arrangement given by the union of $t$ distinct lines $\mathbb{K}^m$. We have that 
\[ H^e(J_\cA) = \sigma^m - t \sigma + \texttt{lower degree terms}.\]
\end{Proposition}

\begin{proof}
Consider the complex 
\[ 0 \to \prod_{i\in A} J_i(V) \to \bigoplus_{|B|=t-1} J_B(V) \to \cdots \to J_1(V) \oplus J_2(V) \oplus \cdots \oplus J_t (V) \to R \to 0.\]
As the $k$-th homology of the complex is concentrated in degree $k$, even though the complex is not exact, we have that up to low degree terms 
\[ 0 \approx \sum (-1)^{|B|} H^e(J_{\cB}) \]
so that 
\[(-1)^{t+1} H^e(J_{\cA}) \approx \sum_{B \subsetneq A} (-1)^{|B|} H^e(J_{\cB})\]

We will prove the claim by induction on $t$. For $t=0$, we have that $H^e(J_{\emptyset})= H^e(R) = \sigma^m$, so the claim holds.
 
Assuming the result for all $\cB$ such that $|B|<t$, we have that 
\[(-1)^{t+1} H^e(J_{\cA}) \approx \sum_{B \subsetneq A} (-1)^{|B|} (\sigma^m - |B| \sigma) . \]

Let $k$ be a non-negative number. The number of $k$-subsets in $A$ is $ \binom{t}{k}$. Thus, we have to prove that
\[\sum_{k<t} (-1)^{k} \binom{t}{k} = (-1)^{t+1}\]
and
\[\sum_{k<t} (-1)^{k} \binom{t}{k} (-k) = (-1)^{t+1}(-t).\]
Consider the generating function $f(x) = (1-x)^t = \sum_{k \leq t} (-1)^{k} \binom{t}{k}  x^k$. Firstly, for $x=1$ we have that
\[ f(1) = 0 = \sum_{k \leq t} (-1)^{k} \binom{t}{k}, \]
so $(-1)^{t+1} = -(-1)^t = \sum_{k<t} (-1)^{k} \binom{t}{k}$, as required. Secondly, consider $f'(1)$:
\[ f'(1) = 0 = \sum_{k \leq t} (-1)^{k} \binom{t}{k} k. \]
Rearranging, we get that $(-1)^{t+1}t = -(-1)^t t = \sum_{k<t} (-1)^{k} \binom{t}{k} k$. Multiplying both sides by $-1$, we get the required equality.

Therefore, by induction, the original result holds.
 \end{proof}

We mentioned that the equivariant Hilbert series of the product ideal can be computed purely from the combinatorial structure of $\cA$, meaning that the information provided by the polymatroid of $\cA$ is sufficient to determine $H^e(J(V))$. 

\begin{Definition}
Let $\cA = \{ W_1, \ldots, W_t \} \subset W$ and consider the index set $A = \{1, \dots , t\}$.  We associate to a subset $B = \{ a_1, \dots, a_s\} \subset \{1, \dots, t\}$ the subarrangement $\cB = \{ W_{a_1}, \ldots , W_{a_s}\}$. The polymatroid $(A, \rk)$ associated to $\cA$ is the power set $\mathfrak{P}(A)$ with rank function $rk : \mathfrak{P}(A) \to \NN $ defined on $B \subseteq A$ by
\[\rk(B) = \dim(W) - \dim(\cap_{i\in B} W_{i}).\]
\end{Definition}

Notice that a polymatroid where all one element subsets have rank one is a matroid. To each subspace arrangement $\cA$ we can associate a symmetric polynomial $P(A)$ associated to the polymatroid $(A,rk)$. We have that $P(A)$ measures how far is the complex $\cC$ from being exact. In particular, if $P(A)=0$, we will have that the complex $\cC$ is exact.

\begin{Definition} 
Define the symmetric polynomial $P(A)$ is of degree $\leq t -1$ recursively as follows. Set $P(\emptyset) =1$. Then
\[ P(A) = u_0 + \cdots + u_{|A|-1}\] 
when 
\[ \sum u_i = - \sum_{B \subsetneq A} (-1)^{|A|-|B|}\sigma^{rk(A)-rk(B)} P(B). \]
\end{Definition}

A result of Derksen gives us a way to compute the equivariant Hilbert series of $J_{\cA}$ from the polymatroid of $\cA$ via the symmetric polynomial $P(A)$.
\begin{Theorem}[Derksen, Theorem 5.2 in \cite{sym}]\label{comp}
Let $\sigma = \sum s_i$. We have that 
\[ \sigma^{(n-rk(A))} P(A) = \sum_{B \subseteq A} (-1)^{|B|} H^e(J_B) \]
\end{Theorem}

\begin{Example}
Consider the two coordinate axes in a two-dimensional vector space $W$. As linear subspaces we can characterize the $x$-axis $W_1$ as the subspace $\{(x,y) \mid y = 0 \}$ and similarly, the $y$-axis $W_2$ as the subspace $\{ (x,y) \mid y = 0 \}$. Then $\cA = \{W_1, W_2 \}$ is a subspace arrangement. The associated linear ideals are $J_1 = (y)$, $J_2 = (x)$, whilst the product ideal is $J_A = (x)(y) = (xy)$.

The polymatroid data $(A, \rk)$ associated to $\cA$ is given below together with the resulting symmetric polynomials $P(B)$ for $B \subseteq A$
\begin{align*}
\rk(\emptyset) = 0 , & \quad P(\emptyset) = 1 \\
\rk(\{i\}) = 1 , & \quad P(\{i\}) = 1  \\
rk(\{1,2\}) = 2 , & \quad P(\{1,2\}) = P(A) = 1
\end{align*}
Using this data, we get that $H^e(\ja) = \sigma^2 - 2\sigma + 1$.
\end{Example}

Let $\cA$ be a subspace arranement of cardinality $t$. Notice that the product of $t$ linear ideals is generated in degree $t$ and it is $t$-regular by Conca and Herzog's result. Thus, we notice that the minimal free resolution of its associate polynomial functor $\ja$ is a linear resolution. Then, we will have that the torsion module $E_i = \Tor_i(\ja, \KK)$ is the only torsion module in the resolution of degree $i+t$. The following result gives us a way to find the signed sum of the characters of the torsion modules from the equivariant Hilbert series $H^e$.

\begin{Corollary}[Derksen, Corollary 5.3 in \cite{sym}]\label{tor}
Let $\sigma = \sum s_i$, let $\cA$ be a subspace arrangement in $W \cong \KK^m$, and let $H^e(\ja)$ be the equivariant Hilbert series of the product ideal functor $\ja$. Let \[ \sigma^{-m} H^e(\ja) = \sum_{\lambda, |\lambda| \geq t} (-1)^{|\lambda| - t} a_{\lambda} s_{\lambda}, \]
where $a_{\lambda} \in \ZZ_{\geq 0}$ is the multiplicity of $s_{\lambda}$. Then, for $E_d = \Tor_d(\ja, \KK)$, we have that
\[ H^e(E_d) = \sum_{|\lambda| = d} a_{\lambda} s_{\lambda}. \]
\end{Corollary}

Corollary \ref{tor} holds because there is only one torsion module $E_d$ of each degree $d$. Therefore, we will have that each $s_\lambda$ in $\sigma^{-m} H^e$ with $|\lambda| = d$  will give us precisely a Schur functors $\Sl$ appearing in $E_d$.

\subsection{Example: powers of the maximal ideal}
Consider the subspace arrangement $\cA =\{ Y_1, \ldots, Y_t \} \subset \KK$ with all $Y_i = \{0\}$. This subspace arrangement is $t$ copies of the zero dimensional subspace in a vector space $Y$ of dimension one. The equivariant product ideal of this subspace arrangement is the $t$-th power of maximal ideal $\cM(V) = (x_1, \ldots, x_n)$ for $V$ an $n$-dimensional vector space. So we have $\cM^t(V)$ in $\cS(\wv)\cong \cS(V)$. From ~\cite{ds1} we know that $\cM ^t(V)$ is $t$-regular and from Theorem~\ref{reg} we can establish that $\W(\cM ^t)(V)$ is also $t$-regular. 

Applying Theorem \ref{comp} to $H^e(\cM ^t)$ we get a minimal resolution for $\cM ^t$.
Let $\sigma = \sum_i s_i = 1 + s_1 + s_2 + \ldots$ and notice that $H^e(\cS)= \sigma$. For $t=1$, we get that $H^e(\cM ) = \sigma - 1$ , yielding the Koszul resolution for $\cM$:
\[ \cdots \to \cS \otimes \cS_{(1,1,1)} \to \cS \otimes \cS_{(1,1)} \to \cS \otimes \cS_{(1)} \to \cM \to 0, \] 
and the following Koszul resolution for $\W(\cM)(V)$:
\[ \cdots \to \we \otimes \cS_{(3)} \to \we \otimes \cS_{(2)} \to \we \otimes \cS_{(1)} \to \W(\cM) \to 0. \]

For $t=2$, we get that $H^e(\cM ^2) = \sigma - (1+ s_1)$. This gives the following resolution for $\cM^2$:
\[ \cdots \to \cS \otimes \cS_{(2,1,1)} \to \cS  \otimes \cS_{(2,1)}  \to \cS  \otimes \cS_{(2)} \to \cM^2 \to 0, \] 
yielding the following resolution for $\W(\cM^2)$:
\[ \cdots \to \we \otimes \cS_{(3,1)} \to \we \otimes \cS_{(2,1)} \to \we \otimes \cS_{(1,1)} \to \W(\cM^2) \to 0. \]

\subsection{Example: distinct lines in a plane}
Consider the subspace arrangement $\cA$ given by $t$ distinct lines in a vector space $Y$ of dimension two. We have that $\cA =\{ Y_1, \ldots, Y_t \} \subset \KK^2$ and where for each $i$ the subspace $Y_i$ is a line in $\KK^2$, such that all lines $Y_i$ are distinct.  By Proposition \ref{lines}, we have that
\[ H^e(\ja) = \sigma^2 - t \sigma - Q(A) ,\]
where $Q(A)$ is a symmetric polynomial of degree less than $t$ depending on the symmetric polynomial $P(A)$ of the polymatroid $(A, \rk)$.

In particular, for $t=2$ we have that $\cA$ consists of two lines in $\KK^2$. We can assume these two lines to be the $y$-axis the $x$-axis, so that $J_1 J_2 = (x,y)$ in $\KK[x,y]$. For each vector space $V$ of dimension $n$, we get the product ideals
\[\ja(V) =  J_1(V) J_2(V) = (x_1,\ldots ,x_n)(y_1, \ldots, y_n) = (x_i y_j) \] 
and 
\[\W(\ja)(V) =  J_1(V) \wedge J_2(V) = (x_1,\ldots ,x_n) \wedge (y_1, \ldots, y_n) = (x_i \wedge y_j) \] 
where $1 \leq i,j \leq n$.

Using the polymatroid of $\cA$, we get that $H^e(\ja) = (\sigma - 1)^2$. Using this formula for the equivariant Hilbert series of $\ja$ we get the following resolution:
\[ \cdots \to  \cS \otimes (\mathcal{S}_{(2,2)} \oplus \mathcal{S}_{(2,1,1)}^3 \oplus \mathcal{S}_{(1,1,1,1)}^3) \to \cS \otimes (\mathcal{S}_{(2,1)}^2 \oplus \mathcal{S}_{(1,1,1)}^2) \to \cS \otimes (\mathcal{S}_{(2)}\oplus \mathcal{S}_{(1,1)}) \to   \mathcal{J}_{\cA} \to 0, \]

yielding the following resolution for $\W( \ja)$:
\[ \cdots \to \we \otimes (\mathcal{S}_{(2,2)} \oplus \mathcal{S}_{(3,1)}^3 \oplus \mathcal{S}_{(4)}^3) \to \we \otimes (\mathcal{S}_{(2,1)}^2\oplus \mathcal{S}_{(3)}^2)) \to \we \otimes (\mathcal{S}_{(1,1)}\oplus \mathcal{S}_{(2)}) \to \Omega( \mathcal{J}_{\cA}) \to 0. \] 	

\subsection{Example: a line and a plane}
Consider the subspace arrangement $\cA$ given by a plane and a line normal to it in an ambient space of dimension three. We have that $\cA =\{ Y_1, Y_2 \} \subset \KK^3$, where we can assume that $Y_1$ is the $(x,y)-$ plane and that $Y_2$ is the $z$-axis. Then in $\KK[x,y,z]$ we have that $J_1= (z)$ and $J_2 = (x,y)$, so that $J = J_1 J_2 = (zx, zy)$. We get that
\[ H^e(\ja) = \sigma^3 - \sigma^2 - \sigma + 1. \]
Using the formula above for the equivariant Hilbert series of $\ja$ we get the following resolution:
\begin{multline*}
\cdots \to  \cS \otimes (\cS_{(3,1)}^3 \oplus \cS_{(2,2)}^5 \oplus \cS_{(2,1,1)}^{12} \oplus \cS_{(1,1,1,1)}^9) \to \\ 
\to \cS \otimes ( \cS_{(3)} \oplus \cS_{(2,1)}^6 \oplus \cS_{(1,1,1)}^5) \to \cS \otimes (\cS_{(2)}^2 \oplus \cS_{(1,1)}^2) \to  \ja \to 0,
\end{multline*}
yielding the following resolution for $\W( \ja)$:
\begin{multline*}
\cdots \to  \we \otimes (\cS_{(2,1,1)}^3 \oplus \cS_{(2,2)}^5 \oplus \cS_{(3,1,1)}^{12} \oplus \cS_{(4)}^9) \to \\ 
\to \we \otimes ( \cS_{(1,1,1)} \oplus \cS_{(2,1)}^6 \oplus \cS_{(3)}^5) \to \we \otimes (\cS_{(1,1)}^2 \oplus \cS_{(2)}^2) \to  \W(\ja) \to 0.
\end{multline*}

\subsection{Example: three coordinate axes}
Consider the subspace arrangement $\cA$ given by the three coordinate axes in $\KK^3$. We have that $\cA =\{ Y_1, Y_2, Y_3 \}$, where $J_1= (y,z)$, $J_2 = (x,z)$, and $J_3 = (x,z)$ in $\KK[x,y,z]$. Consider in this case the \emph{intersection} ideal $I_{\cA} = J_1 \cap J_2 \cap J_3$. One can check that $I_{\cA} = (xy, xz, yz)$. Being generated in degree two, we know that the regularity of $I_{\cA}$ is at least two. One can check that $I_{\cA}$ is the sum of three products of linear ideals, 
\[ I_{\cA} = (xy) + (xz) + (yz)= (xy,xz,yz), \]
so that we can use a result of Derksen and Sidman \cite{ds2} to conclude that $I_{\cA}$ is 4-regular. 

We can show that $I_{\cA}$ has regularity two. In fact, we will show that $\ia (V)$ has regularity two for every vector space $V$, so that the functor $\ia$ has regularity two. Let us use the notation $\mathbf{x} = x_1, \ldots, x_n$, and similarly for $\mathbf{y}, \mathbf{z}$.

\begin{Proposition}
Let $V$ be a vector space of dimension $n$. Let $J_1(V)=(\mathbf{y}, \mathbf{z}), J_2(v)= (\mathbf{x},\mathbf{z}), J_3(V)= (\mathbf{x}, \mathbf{y})$. Then the ideal 
\[ \ia(V) = J_1(V) \cap J_2(V) \cap J_3(V) = (\mathbf{y}, \mathbf{z}) \cap (\mathbf{x},\mathbf{z}) \cap (\mathbf{x}, \mathbf{y})\]
has regularity two.
\end{Proposition}

\begin{proof}
We have that $\ia(V)$ is generated in degree at least two. In fact, one can check that 
\[\ia(V) = (x_1 y_1, \ldots , x_i y_j, \ldots, x_n y_n) + (x_1 z_1 , \ldots, x_i z_k , \ldots, x_n z_n) + (y_1 z_1, , \ldots , y_j z_k , \ldots, y_n z_n),  \]
where $1 \leq i,j,k \leq n$. As $\ia(V)$ is generated in degree two, it has regularity at least two. So to prove the proposition it is enough to show that the regularity is at most two.

Let $\KK[X]$ be the coordinate ring of $X = Y_1 \otimes V$. We have that $X$ is the subspace spanned by the $x$ coordinate axes 
\[ \KK[X] = \KK[\mathbf{x}, \mathbf{y}, \mathbf{z}] / J_1(V) = \frac{\KK[\mathbf{x}, \mathbf{y}, \mathbf{z}]}{(\mathbf{y}, \mathbf{z})}. \] 
Similarly, $\KK[Y] = \KK[\mathbf{x}, \mathbf{y}, \mathbf{z}] / J_2(V)$ and $\KK[Z] = \KK[\mathbf{x}, \mathbf{y}, \mathbf{z}] / J_3(V)$. Let $\pi_i$ be the canonical projection $\pi_i: \KK[\mathbf{x}, \mathbf{y}, \mathbf{z}] \to \KK[\mathbf{x}, \mathbf{y}, \mathbf{z}]/J_i(V)$. We then have a map $\pi = (\pi_1, \pi_2, \pi_3)$ such that 
\[ \pi: \KK[\mathbf{x}, \mathbf{y}, \mathbf{z}] \to \KK[X] \oplus \KK[Y] \oplus \KK[Z]. \]

We have that $\pi$ factors through $\KK[\cA] = \KK[\mathbf{x}, \mathbf{y}, \mathbf{z}]/ \ia(V)$ giving us a map $\phi$:
\[ \phi: \KK[\cA] \to \KK[X] \oplus \KK[Y] \oplus \KK[Z]. \]
Notice that an element of $r \in \KK[\cA]$ can be represented as $ r = f(\mathbf{x}) + g(\mathbf{y}) + h(\mathbf{z})$. Hence $\phi$ is given by $f(\mathbf{x}) + g(\mathbf{y}) + h(\mathbf{z}) \mapsto (f(\mathbf{x}), g(\mathbf{y}),  h(\mathbf{z}))$. The image of $\phi$ is 
\[ U = \{ (f,g,h) : f(0)= g(0) = h(0) \} \]
as the intersection of any two of the three subspaces is the origin and so $X \cap Y \cap Z$ is also $0$. Thus, $(f,g,h)$ determines a function on $\cA$, the union $X \cup Y \cup Z$, precisely when the functions $f,g,h$ agree at $0$.

Consider the short exact sequence of modules
\[ 0 \to U \to \KK[X] \oplus \KK[Y] \oplus \KK[Z] \to \KK^2 \to 0, \]
where the last map is the projection $(f,g,h) \mapsto (f(0)-g(0), g(0)-h(0))$. Consider the following result on regularity (part of Corollary 20.19 in \cite{ca}).

\begin{Lemma}\label{le}
If $A,B,C$ are finitely generated graded modules, and
\[0 \to A \to B \to C \to 0\]
is exact, then $\reg(A) \leq \max \{ \reg(B), \reg(C)+1\}$.
\end{Lemma}

Thus, we can bound the regularity of the first module $A$ in a short exact sequence if we know bound on the regularity of the other two modules. In particular, we have that 
\[ \reg (U) \leq \max \{ \reg(\KK[X] \oplus \KK[Y] \oplus \KK[Z]), \reg(\KK^2) +1 \} = 1 .\]

Finally, consider the short exact sequence of modules
\[ 0 \to \ia(V) \to \KK[\mathbf{x}, \mathbf{y}, \mathbf{z}] \to U \to 0, \]
where the last map is the restriction of the projection $\pi$. Using Lemma~\ref{le} again, we conclude that
\[ \reg (\ia(V)) \leq \max \{ \reg(\KK[\mathbf{x}, \mathbf{y}, \mathbf{z}]), \reg(U) +1 \} \leq 2 .\]
As $\reg (\ia(V)) \geq 2$, we can conclude that $\reg (\ia(V))= 2$.
\end{proof}

As $\ia(V)$ is generated in degree two and has regularity two, we have that $\ia$ has a linear resolution. Thus, we can use its equivariant Hilbert series to write down its resolution. We have that
\[ H^e(\ia) = \sigma^3 - (3\sigma -2), \]
yielding the following resolution for $\ia$:
\begin{multline*}
\cdots \to  \cS \otimes (\cS_{(3,1)}^6 \oplus \cS_{(2,2)}^9 \oplus \cS_{(2,1,1)}^{21} \oplus \cS_{(1,1,1,1)}^{15}) \to \\ 
\to \cS \otimes ( \cS_{(3)}^2 \oplus \cS_{(2,1)}^{10} \oplus \cS_{(1,1,1)}^8) \to \cS \otimes (\cS_{(2)}^3 \oplus \cS_{(1,1)}^3) \to  \ia \to 0,
\end{multline*}
and for $\W(\ia)$:
\begin{multline*}
\cdots \to  \we \otimes (\cS_{(2,1,1)}^6 \oplus \cS_{(2,2)}^9 \oplus \cS_{(3,1)}^{21} \oplus \cS_{(4)}^{15}) \to \\ 
\to \we \otimes ( \cS_{(1,1,1)}^2 \oplus \cS_{(2,1)}^{10} \oplus \cS_{(3)}^8) \to \we \otimes (\cS_{(1,1)}^3 \oplus \cS_{(2)}^3) \to  \W(\ia) \to 0.
\end{multline*}

\section{Conclusion}

In this paper we have discussed an application of the transpose functor to a class of equivariant ideals over the symmetric algebra. The transpose functor technique can be applied to other classes of equivariant modules to translate homological properties of modules over the symmetric algebra to properties of modules over the exterior algebra. In the future we plan to study super vector spaces to extend these results to modules over super algebras.

Moreover, we have seen how the combinatorial information of the polymatroid associated to a subspace arrangement determines the equivariant resolution of the product ideal associated to the subspace arrangement. We hope that further study of the combinatorial properties of equivariant Hilbert series will provide more insights into the homological properties of these ideals.

Finally, we have mentioned in the introduction a connection to non-commutative invariant theory. In her thesis \cite{fg}, the author proves that the regularity bound on the intersection ideal of a subspace arrangement over the exterior algebra implies a degree bound on the minimal generating invariant skew polynomials of any finite group. In particular, an analog of Noether's Degree Bound \cite{em} is true for the exterior algebra. These results will appear in a forthcoming paper.


\end{document}